\documentclass[12pt]{article}

\usepackage{amstext,amsmath,amssymb,amsfonts,bbm}
\usepackage[latin1]{inputenc}
\usepackage{epsfig}
\usepackage{dsfont}
\usepackage{hyperref}
\usepackage{amsthm}
\usepackage{subfigure}
\usepackage{color}
\usepackage{multirow}
\usepackage{psfrag}
\usepackage{graphicx}
\usepackage{extarrows}

\theoremstyle{plain}

\newtheorem{lemma}{Lemma}

\newtheorem{theorem}{Theorem}

\newtheorem{definition}{Definition}

\theoremstyle{definition}
\newtheorem{example}{Example}

\newcommand{\beq}{\begin{equation}}
\newcommand{\eeq}{\end{equation}}
\newcommand{\beqa}{\begin{eqnarray}}
\newcommand{\eeqa}{\end{eqnarray}}

\def\NN{{\mathbb N}}

\def\NN{{\mathbb N}}

\begin{document}

\title{Using Grassmann calculus in combinatorics: Lindstr\"{o}m-Gessel-Viennot lemma and Schur functions}

\author{S. Carrozza\footnote{Labri, Univ. Bordeaux (sylvain.carrozza@labri.fr)}, T. Krajewski\footnote{CPT, Aix-Marseille Univ. (krajewsk@gmail.com)}, A. Tanasa\footnote{Labri, Univ. Bordeaux, Talence $\&$ NIPNE Magurele $\&$ IUF (adrian.tanasa@labri.fr)}}
\maketitle

\abstract{Grassmann (or anti-commuting) variables are extensively used in theoretical physics.
%in relation with fermions (which are particles obeying anti-commutation laws).
In this paper we use Grassmann variable calculus to give new proofs of celebrated combinatorial identities such as the Lindstr\"{o}m-Gessel-Viennot formula for graphs with cycles and the Jacobi-Trudi identity. Moreover, we define a one parameter extension of Schur polynomials that obey a natural convolution identity.}

%%%%%%%%%%%%%%%%%%%%%%%%%%%%%
%\section{Introduction}
%%%%%%%%%%%%%%%%%%%%%%%%%%%%%

%%%%%%%%%%%%%%%%%%%%%%%%%%%%%
\section{Introduction - Grassmann variables and calculus}
%%%%%%%%%%%%%%%%%%%%%%%%%%%%%

Grassmann (or anti-commuting) variables $\chi_1, ..., \chi_m$ ($m\in\NN)$
% satisfy
are defined through their anti-commutation relations:
\beqa
\label{defg}
 \chi_i \chi_j = - \chi_j \chi_i, \quad \forall  i,  j=1,\ldots , m.
\eeqa
As an immediate consequence, one has the following crucial identity:
\beqa
\label{defg2}
 \chi_i^2= 0, \quad \forall  i=1,\ldots , m.
\eeqa
More precisely, the Grassmann algebra $\Lambda_m$ over the $m$ anti-commuting variables $\{\chi_1, \ldots, \chi_m\}$ is defined as the linear span of the $2^m$ independent products of the $\chi_i$'s. Its elements are functions of the form
\beq
f(\chi) = \sum_{n=0}^{m} \frac{1}{m!}\sum_{1 \leq i_1, \ldots \, i_n \leq n} a_{i_1\ldots i_n} \chi_{i_1} \ldots \chi_{i_n},
\eeq
where $a_{i_1 \ldots i_n}$ are complex coefficients, antisymmetric with respect to their indices, $a_{i_{\sigma(1),\dots i_{\sigma(n)}}}=\epsilon(\sigma) a_{i_{1},\dots,i_{n}}$, and the vector space structure is simply defined by addition and scalar multiplication of the coefficients. A function $f$ which is a sum of only even (resp. odd) monomials is called { even} (resp. odd). The multiplication rule for monomials is 
\begin{equation}\label{multiplication}
(\chi_{i_1} \ldots \chi_{i_n}) (\chi_{j_1} \ldots \chi_{j_p}) = \left\{ 
\begin{split} &0 \qquad &\mathrm{if} \quad \{ i_1, \ldots , i_n \} \cap \{ j_1, \ldots , j_p \} \neq \emptyset \\
 &\mathrm{sgn}(k) \, \chi_{k_1} \ldots \chi_{k_{n+p}} & \mathrm{otherwise} 
\end{split}\right.\, ,
\end{equation}
with $k = (k_1, \ldots , k_{n+p})$ the permutation of $(i_1, \ldots , i_n, j_1 , \ldots , j_{p})$ such that $k_1 < \ldots < k_{n+p}$. It is then extended to the whole Grassmann algebra by distributivity. For notational convenience, we will furthermore commute complex and Grassmann variables, and to permute the Grassmann variables themselves following the defining rule \eqref{defg}.

%We see that, in accordance with definition \eqref{defg}, any function of Grassmann variables is a polynomial of bounded degree. 
One then defines the exponential of a Grassmann function $f$ by
\begin{equation}
\label{expg}
   e^{f(\chi)} := \sum_{p = 0}^{+ \infty} \frac{1}{p!} f(\chi)^p,
\end{equation}
which, following \eqref{defg}, is a simple polynomial expression. In particular one immediately finds that 
%\beq
$e^{\chi_{i_1} \ldots \chi_{i_n}} = 1 + \chi_{i_1} \ldots \chi_{i_n}.$
%\eeq
Another interesting property is that $e^{A}e^B = e^{A+B}$ for any even Grassmann functions $A$ and $B$ (since $A$ and $B$ therefore commute). 

Due to these multiple properties, Grassmann variables are extensively used in quantum field theory to describe the physics of fermions\footnote{Example of fermions are: the electrons, the neutrinos, the quarks.}, which are particles obeying the so-called Fermi-Dirac statistics, statistics which is based on anti-commutation laws (unlike bosons\footnote{Examples of bosons are: photons, gluons, Higgs bosons.}, which are particles obeying the Bose-Einstein statistics (statistics based on commutation laws), and which are described by physicists using usual commuting variables) -- the interested reader is reported to quantum field theory textbooks such as \cite{QFT-ZJ} for more details.

\medskip
The {\bf Grassmann integral} $\int d\chi \equiv \int d\chi_m \ldots d\chi_1 $ is the unique linear map from 
$\Lambda_m$ to $\mathbb{C}$ s. t.
%\beqa
%\label{intg}
  $$\int d\chi \, \chi_1 \ldots \chi_m = 1.$$
Moreover, $\int d\chi \, \chi_{i_1} \ldots \chi_{i_n} = 0 \quad \mathrm{whenever} \quad n < m$.
%\eeqa
%\beqa
%  \int d\chi = 0,  \; \;    \;  \;  \int (d\chi) \chi = 1.
%\eeqa
\begin{example}
Let $\chi$ and $\bar \chi$ be two independent Grassmann variables (the bar has nothing to do with any complex conjugation) and let $a \in \mathbb{C}$. One computes:
\begin{equation}
    \int d \bar \chi d \chi \, 
    e^{-\bar \chi a \chi}=
    \int d \bar \chi d \chi 
    \, (1-\bar \chi a \chi) =
     \int d \bar \chi d \chi 
    \, (-\bar \chi a \chi)=
    a \int d \bar \chi d \chi 
    \, \chi \bar \chi =
    % \int d \bar \chi  
    %( a \bar \chi)=
     %a \int d \bar \chi  
    %(\bar \chi)=
    a .
\end{equation}
Similarly, one computes:
%\begin{equation}
 $     \int d \bar \chi d \chi \, 
    \chi \bar \chi \, e^{-\bar \chi a \chi}
    = 1.$
%\end{equation}
\end{example}
\begin{example}
Consider $N$ independent Grassmann variables $\{ \chi_i\, \vert \, 1 \leq i \leq N\}$. Then, for any permutations $\sigma$ of $\{ 1,\ldots, N\}$, one has:
\beq\label{rule_int}
\int d\chi_N \ldots d\chi_1 \, \chi_{\sigma(1)} \ldots \chi_{\sigma(N)}= \mathrm{sgn}(\sigma).
\eeq
\end{example}
%An important fact about Grassmann calculus is:
%
%\begin{theorem}
%\label{teorema}
%Let  $2m$ independent
%Grassmann variables which one can denote by 
%\begin{equation}
%\bar \chi_1, \ldots , \bar \chi_m, \chi_1, \ldots,  \chi_m
%\end{equation}
% (where the bars have, again, nothing to do with complex conjugation) and let  $f(\bar \chi_1, \ldots; \bar \chi_m, \chi_1, \ldots , \chi_m)$ some function of these $2m$ Grassmann variables.
%In the Grassmann integral
%\begin{equation}
%    \int \left( \prod_{i=1}^m d \bar \chi_i d \chi_i \right) f(\bar \chi_1, \ldots, \bar \chi_m, \chi_1, \ldots , \chi_m)
%\end{equation}
%the only term leading to a non-vanishing contribution is the term having {\bf exactly one} factor $\bar \chi_i$ and 
%{\bf exactly one} factor $\chi_i$ ($i=1,\ldots, 2m$).
%\end{theorem}
%\begin{proof}
%This is a direct consequence of the Grassmann calculus rules \eqref{intg}. \qed
%\end{proof}

%Gaussian integration in $N$ Grassmann dimensions allows to easily compute $N \times N$ Pfaffians. 
Let $M$ be an $N-$dimensional square matrix whose entries are commuting variables (such as complex numbers). 
Its {\bf determinant} 
 can be expressed as
a Grassmann Gaussian integral over $2N$ Grassmann variables $\bar \chi_i,\chi_i$, $i=1,\ldots, N$. As above, the conjugate notation conveniently accounts for the doubling of variables.
Using the morphism property of the exponential on even functions, one 
%has:
proves:
\beqa\label{det}
\det M = \int d\bar \chi_N d\chi_N \ldots d\bar \chi_1 d\chi_1 \,  \exp\left(-\sum_{i,j=1}^N \bar\chi_i M_{ij} \chi_j  \right).
\eeqa
Similarly, one can express any {\bf minor} of $M$ using Grassmann calculus.
Let $0\le p \le N$ and let $I=\{ i_1, \ldots, i_p\}$, $J=\{ j_1,\ldots, j_p\}$ be two subsets of indices of $\{1,\ldots, N\}$, where $i_1<\ldots <i_p$ and $j_1<\ldots <j_p$.
We denote by $M_{I^c J^c}$ the matrix obtained by deleting from $M$ the rows with indices in $I$ and the column with indices in $J$. One has
\begin{equation}
\label{minorI}
    (-1)^{\Sigma I + \Sigma J} \, \det (M_{I^c J^c})=
    \int d \bar \chi d \chi \, \chi_J \bar \chi_I \, \exp\left( -\sum_{i,j=1}^N \bar\chi_i M_{ij} \chi_j \right),
\end{equation}
where we have used the notation
$$d \bar \chi d \chi := d\bar \chi_N d\chi_N \ldots d\bar \chi_1 d\chi_1.$$ 
Moreover, $\Sigma I =\sum_{k=1}^p i_k$, $\Sigma J =\sum_{k=1}^p j_k$ and 
\begin{equation}
    \chi_J \bar \chi_I := \chi_{j_1} \bar \chi_{i_1}\ldots \chi_{j_p} \bar \chi_{i_p}\,.
\end{equation}
%One can obviously relabel the set of rows and column to be deleted, such that the formula above simplifies to:
%\begin{equation}
%\label{minor}
%    det\, M_{IJ}=\int \prod_{i=1}^n \left( d\bar \chi_i d\chi_i \right) \left( \chi_J \bar \chi_I \right) e^{-\sum_{i,j=1}^n \bar\chi_i M_{ij} \chi_j}
%\end{equation}

%In the following, we will sometimes introduce a second copy of Grassmann variables $\{ \bar \chi_i,\chi_i \, \vert \, i=1,\ldots, N \}$.
%, which in physics assume the role of external sources. 
%Together with a more complete discussion of Grassmann integration (including changes of variables), this 
%allows to consider and 
One can 
prove more general {\bf Grassmann Gaussian integral} formulas\footnote{See, for example, \cite{QFT-ZJ}, for more details on Grassmann integration and Grassmann changes of variables.}, such as:
\begin{equation}
\label{gausiana}
     \int d \eta d \bar \eta\, \exp\left( \sum_{k,\ell=1}^N\bar \eta_k M^{-1}_{k\ell}
      \eta_\ell + \sum_{k=1}^N (\bar \psi_k \eta_k
      + \bar \eta_k \psi_k ) \right) = 
    \, \det(M^{-1})\, \exp\left( - \sum_{k,\ell=1}^N
    \bar \psi_k M_{k\ell} \psi_\ell \right).
\end{equation}
The integrand here leaves 
%in the tensor product of the two Grassmann algebras (or, equivalently, 
in the Grassmann algebra $\{ \bar \eta_i, \eta_i, \bar \chi_i,\chi_i \, \vert \, i=1,\ldots, N \}$. 

%These properties (and many others) lead to very interesting developments outside the original field of theoretical physics. Thus, in combinatorics, Grassmann calculus has been already used to prove: 
%the contraction-deletion formula of the Tutte and Bollob\'as-Riordan polynomials \cite{KRTW}, matrix-tree-like theorems \cite{malek} and so on.

%For various other using these Grassmann techniques, the interested reader may consult Abdelmalek Abdesselam's paper \cite{malek} or Sergio Caracciolo's, Alan D. Sokal and A. Sportiello's paper \cite{css}.

In this paper, we use Grassmann calculus to prove the  Lindstr\"{o}m-Gessel-Viennot (LGV) lemma and the  Jacobi-Trudi identity.
Note that the LGV lemma proof we give does not require the use of any involution arguments. The vanishing contribution of intersecting paths will appear as a simple consequence of the Grassmann nilpotency identity \eqref{defg2}! 
%(see the following section)!

%%%%%%%%%%%%%%%%%%%%%%%%%%%%%
\section{Grassmann proof of the LGV lemma}
%%%%%%%%%%%%%%%%%%%%%%%%%%%%%

%Let us first recall the tLindstr\"{o}m-Gessel-Viennot formula (see \cite{L}, \cite{GV}). 
Let $G$ be a finite directed graph. Note that 
%we do not allow for multiple edges, but 
we allow loops and multiple edges.
%; we will accordingly prove a generalization to not necessarily acyclic graphs of the original formula, first proved (using standard combinatorial techniques) in \cite{talaska}. 
Let $V=\{ v_1, \ldots, v_N\}$ be the set of vertices of $G$. One assigns to each edge $e$ a {\bf weight} $w_e$. One further assumes that the variables $w_e$ commute with each others. 

A {\bf path} $P$ from $v$ to $v'$ 
%(denoted $P: v \to v'$) 
is a collection of edges $(e_1, e_2, \ldots , e_k)$ such that one can reach $v'$ from $v$ by successively traversing $e_1, \ldots, e_k$ in the specified order. 
%We write $P:v ? v'$ to indicate the $P$ is a path starting at a vertex $v$ and ending at a vertex $v'$. 
Following \cite{talaska}, let us recall the following definitions. 
%\begin{definition}
The {\bf weight} of a given path $P=(e_1,\ldots , e_k)$  is: 
\beq
\mathrm{wt}(P):=\prod_{k=1}^{m} w_{e_k}\,.
\eeq
%\end{definition}
%\begin{definition}
The {\bf weight path matrix} of the graph $G$ is the matrix $M= (m_{ij})_{1\leq i,j\leq N}$,  whose entries are:
\beq
m_{ij}:=\sum
%_{P: v_i \to v_j} 
\mathrm{wt}(P).
\eeq
The sum above is taken on paths $P$ from $v_i$ to $v_j$.
%\end{definition}

These quantities are considered as formal power-series in the weights.
%, which can however be analytically defined provided that the weights are small enough. 
A crucial remark is that
\begin{equation}\label{ma}
M=(\mathrm{Id} - A)^{-1},
\end{equation}
where $A = (A_{ij})$ is the {\bf weighted adjacency matrix} of the graph ($A_{ij} = w_{ij}$ if there is an edge from $v_i$ to $v_j$, and $0$ otherwise).

A {\bf cycle} is a path from a vertex $v$ to itself (or more precisely an equivalent class thereof up to change of source vertex). We denote by $\mathcal{C}$ the set of all possible collections of self-avoiding and pairwise vertex-disjoint cycles, including the empty collection. Given $\mathbf{C} = (C_1, \ldots, C_k) \in \mathcal{C}$, we define its weight and sign as
\beq
\mathrm{wt}(\mathbf{C}):= \prod_{i=1}^k \mathrm{wt}(C_i)\,, \qquad \mathrm{and} \qquad \mathrm{sgn}(\mathbf{C}):= (-1)^k,
\eeq
while, by convention, $\mathrm{wt}(\mathbf{C}) = \mathrm{sgn}(\mathbf{C}) =1$ for the empty collection.
\begin{lemma}\label{detM} 
The determinant of $M^{-1}$ is:
\beq
\det (M^{-1}) = \sum_{\mathbf{C} \in \mathcal{C}} \mathrm{sgn}(\mathbf{C}). \mathrm{wt}(\mathbf{C}).
\eeq
\end{lemma}
\begin{proof}
Equations \eqref{ma} and \eqref{det} yield:
\begin{align}\label{detMpf}
\mathrm{det}(M^{-1}) &= \int d\bar \chi d\chi \,  \exp\left(-\sum_{i,j=1}^N \bar\chi_i ( \delta_{ij} - A_{ij} ) \chi_j  
\right)
% \\
%&
= 
\int d\bar \chi d\chi  \, \prod_{i = 1}^N (1 + \chi_i \bar{\chi}_i) \prod_{k,l = 1}^N (1 + A_{kl} \bar\chi_k \chi_l).
\end{align}
The integrand decomposes as sums of terms of the form
\beq
A_{k_1 l_1} \ldots A_{k_s l_s} \times  \bar\chi_{k_1} {\chi}_{l_1}
\ldots \bar\chi_{k_s} {\chi}_{l_s} \; \chi_{i_1} \bar\chi_{i_1} \ldots \chi_{i_r} \bar\chi_{i_r},
\eeq
giving a non-zero contribution to the integral if and only if: all Grassmann variables appear exactly once and $A_{k_j l_j} \neq 0$ for all $1\leq j\leq s$. Let us assume that this is the case. The inequality $A_{k_j l_j} \neq 0$ %for all $1\leq i\leq s$ 
implies that there is a directed edge $e_j$ from $v_{k_j}$ to $v_{l_j}$; this means that $A_{k_j l_j} = w_{k_j l_j}$. 
Let us call $H_s$ the subgraph made out of the edges $e_1, \, \ldots, \, e_s$. Each ingoing (resp. outgoing) edge at a vertex $v_\ell \in H_s$ is associated to a variable $\chi_\ell$ (resp. $\bar\chi_\ell$). Hence there cannot be more than one ingoing (resp. outgoing) edge of $H_s$ at each $v_\ell$. On the other hand, if there were only say one ingoing but no outgoing edge at $v_\ell$, this would require that $\bar\chi_{i_k} = \bar\chi_\ell$ for some $1 \leq k \leq r$. This would however necessarily bring a second factor $\chi_{i_k}$ and therefore cancel the integrand. We conclude that there must be exactly one ingoing at one outgoing edge at each vertex of $H_s$. This means that $H_s$ must decompose into a collection $\mathbf{C}$ of self-avoiding and pairwise vertex-distinct cycles. Furthermore there is in this case a unique choice of indices $1 \leq i_1 < \ldots < i_r \leq N$ yielding a non-vanishing monomial of degree $2N$. Each collection of cycles $\mathbf{C}$ is weighted by $\mathrm{wt}(\mathbf{C})$, up to a sign. Moreover, the integral is of the form of \eqref{rule_int}, with $\sigma$ a product of $\vert\mathbf{C}\vert$ disjoint cycles of even length, the other cycles being trivially of length $1$. The signature of $\sigma$ is therefore $\mathrm{sgn}(\mathbf{C})$, and we conclude that $\mathbf{C}$ contributes with a term $\mathrm{sgn}(\mathbf{C}) \mathrm{wt}(\mathbf{C})$.
\end{proof}

One considers now the minor $M_{{\cal A}{\cal B}}$, where 
${\cal A} = \{ a_1, \ldots , a_p \}$ and $ {\cal B} = \{ a_1, \ldots , a_p \}$ are $p$-dimensional sets of indices in $\{1,\ldots , n\}$. A {\bf p-path} from $\cal A$ to $\cal B$ is a collection of paths $\mathbf{P} = (P_1, \ldots, P_k)$ s. t. $P_i$ connects $a_i$ to $b_{\sigma_{\mathbf{P}} (i)}$, for some permutation $\sigma_{\mathbf{P}}$. The weight and sign of $\mathbf{P}$ are furthermore given by:
\beq
\mathrm{wt}(\mathbf{P}):= \prod_{i=1}^k \mathrm{wt}(P_i)\,, \qquad \mathrm{and} \qquad \mathrm{sgn}(\mathbf{P}):= \mathrm{sgn}(\sigma_{\mathbf{P}}).
\eeq
The p-path 
$\mathbf{P}$ is {\bf self-avoiding} if: 1) each $P_i$ is self-avoiding; 2) $P_i$ and $P_j$ are vertex-disjoint whenever $i \neq j$. We denote by $\mathcal{P}_{{\cal A},{\cal B}}$ the set of  self-avoiding p-paths from $\cal A$ to $\cal B$.
 
Finally, a {\bf self-avoiding flow} from $\cal A$ to $\cal B$ is a pair $(\mathbf{P}, \mathbf{C})$ such that: 1) $\mathbf{P} \in \mathcal{P}_{{\cal A},{\cal B}}$; 2) $\mathbf{C} \in \mathcal{C}$; and 3) $\mathbf{P}$ and $\mathbf{C}$ are vertex disjoint. We denote the set of self-avoiding flow from $\cal A$ to $\cal B$ by $\mathcal{F}_{{\cal A},{\cal B}}$.

The LGV formula for graph with cycles is:
\begin{theorem}
\label{GV:thm}
One has \cite{talaska}
\begin{equation}\label{gv}
    \det (M_{{\cal A} {\cal B}}) =\frac{\underset{(\mathbf{P}, \mathbf{C})\in \mathcal{F}_{{\cal A},{\cal B}}}{\sum} \mathrm{sgn}(\mathbf{P})\mathrm{wt}(\mathbf{P}) \, \mathrm{sgn}(\mathbf{C})\mathrm{wt}(\mathbf{C})}{\underset{\mathbf{C} \in \mathcal{C}}{\sum} \mathrm{sgn}(\mathbf{C}) \mathrm{wt}(\mathbf{C})}.
\end{equation}
In particular, if $G$ is acyclic then \cite{GV, L}:
\beq
\det (M_{{\cal A}{\cal B}}) = \sum_{\mathbf{P} \in \mathcal{P}_{{\cal A},{\cal B}}} \mathrm{sgn}(\mathbf{P})\mathrm{wt}(\mathbf{P}).
\eeq
\end{theorem}
\noindent Let us now give the Grassmann calculus proof of this identity. 
\begin{proof}
The left-hand side of \eqref{gv} is a minor of the matrix $M$. We need to re-express it as a minor of $M^{-1}$. To this purpose, we could directly use 
$\det((M^{-1})_{{\cal A}{\cal B}}) = (-1)^{\Sigma {\cal A} + \Sigma {\cal B}} \, \frac{\det(M_{{\cal B}^c  {\cal A}^c})}{\det (M)}$. Nevertheless, in this paper we instead rely exclusively on Grassmann calculus.
One can thus use formula \eqref{minorI} to express $\det (M_{{\cal A}{\cal B}})$ as
\begin{equation}
      (-1)^{\Sigma {\cal A}^c + \Sigma {\cal B}^c}\,  \int d\bar \psi d\psi  \, \psi_{{\cal B}^c} \bar \psi_{{\cal A}^c} \, \exp\left(-\sum_{i,j=1}^N \bar\psi_i M_{ij} \psi_j \right).
\end{equation}
We now re-express the exponential above using the Grassmann Gaussian integral formula \eqref{gausiana}. This leads to
\begin{equation}
\label{inainte0}
      \frac{(-1)^{\Sigma {\cal A}^c + \Sigma {\cal B}^c}}{\det (M^{-1})}\int d\bar \psi d\psi  \, \psi_{{\cal B}^c} \bar \psi_{{\cal A}^c} \, 
      \int d\eta d\bar\eta  \exp\left(\sum_{k\ell=1}^N\bar \eta_k M^{-1}_{k\ell}
      \eta_\ell + \sum_{k=1}^N (\bar \psi_k \eta_k
      + \bar \eta_k \psi_k ) \right).
\end{equation}
%Note that 
%\begin{equation}
%  det (M^{-1})=1,  
%\end{equation}
%since the graph $G$ is acyclic (see above).
The denominator is given by Lemma \ref{detM}. The integral in the numerator writes: 
%the integrand and the measure can be reorganized as:
\begin{equation}
     \label{inainte}
    (-1)^{\Sigma {\cal A}^c + \Sigma {\cal B}^c }
      \int  d \eta d\bar\eta  \, 
       \exp\left( \sum_{k,\ell=1}^N \bar \eta_k M^{-1}_{k\ell}
      \eta_\ell\right)
      \int d\bar \psi d\psi \,  \psi_{{\cal B}^c} \bar \psi_{{\cal A}^c} \, 
      \exp\left( \sum_k (\bar \psi_k \eta_k
      + \bar \eta_k \psi_k ) \right).
\end{equation}
In order to perform the Grassmann integral on the sets of variables $\bar \psi$ and $\psi$ in \eqref{inainte}, we use the following result:

\begin{lemma}
The following identity holds 
\begin{equation}
    \label{lema}
\int d\bar \psi d\psi \,  \psi_{{\cal B}^c} \bar \psi_{{\cal A}^c} \,
      \exp\left( \sum_{k=1}^N (\bar \psi_k \eta_k
      + \bar \eta_k \psi_k ) \right)
      = (-1)^{\Sigma {\cal A} + \Sigma {\cal B}} \, \bar \eta_{{\cal B}} \eta_{{\cal A}} .
      \end{equation}
\end{lemma}
\begin{proof}
When developing the exponential in \eqref{lema} above, the only term which leads to a non-vanishing contribution is the one containing:
\begin{equation}
    \left ( \prod_{i=1}^p \bar \psi_{a_i} \eta_{a_i}  \right)
     \left ( \prod_{i = 1}^p \bar \eta_{b_i} \psi_{b_i}  \right) = \prod_{i=1}^p \bar \psi_{a_i} \eta_{a_i}  \bar\eta_{b_i} \psi_{b_i} 
     %=  (\bar\psi_{\cal A} \psi_{\cal B}) ( \eta_{{\cal A}} \bar \eta_{{\cal B}} ) 
     =  ( \psi_{\cal B} \bar\psi_{\cal A}) (  \bar \eta_{\cal B}  \eta_{\cal A}).
\end{equation}
The Grassmann integration on the lef-hand side of \eqref{lema} leads to $( \eta_{A} \bar \eta_{B} )$ multiplied by the sign: 
\begin{equation}
\int  d\bar \psi d\psi \,  (\psi_{{\cal B}^c} \bar \psi_{{\cal A}^c}) \, (\psi_{\cal B} \bar\psi_{\cal A} ) = (-1)^{\Sigma {\cal A} + \Sigma {\cal B}}.
\end{equation}
This concludes the proof.
\end{proof}
\medskip

\noindent
Expression \eqref{inainte} above thus becomes:
\begin{equation}
     \label{dupa}
     % \frac{1}{det (M^{-1})}
      \int  d\eta d\bar\eta \,  
        \bar \eta_{{\cal B}} \eta_{{\cal A}} \, 
       \exp\left( \sum_{k,\ell=1}^N \bar \eta_k M^{-1}_{k\ell} 
      \eta_\ell \right) =  \int d\eta d\bar\eta \,  
      \bar \eta_{{\cal B}} \eta_{{\cal A}} \, \prod_{i=1}^N e^{ \bar\eta_i \eta_i } \prod_{k,\ell=1}^N e^{  - \bar \eta_k A_{k\ell} \eta_\ell}.
\end{equation}
%This rewrites as
%\begin{equation}
%      %\frac{1}{det (M^{-1})}
%      \int [ d\bar \eta_k d\eta ] 
%      \, \bar \eta _{B} \eta_{A} \, 
%       \prod_{k,\ell=1}^N e^{\bar \eta_k M^{-1}_{k\ell}
%      \eta_\ell}.
%\end{equation}
Using now \eqref{ma}, this rewrites as
\begin{equation}\label{f}
(-1)^p \int d\bar\eta d\eta  \, \,  
      \bar \eta_{{\cal B}} \eta_{{\cal A}} \, \prod_{i = 1}^N (1 + \eta_i \bar\eta_i) \prod_{k,l = 1}^N (1 + A_{kl} \bar\eta_k \eta_l ).
\end{equation}
%This equation  to equation \eqref{detMpf}, except for the new factor $\bar \eta _{B} \eta_{A}$. 
A similar analysis as the one of Lemma \ref{detM} then shows that the non-zero contributions to the integral are 
labelled by self-avoiding flows $(\mathbf{P}, \mathbf{C}) \in \mathcal{F}_{{\cal A},{\cal B}}$. Indeed, open paths are now allowed, but their source (resp. sink) vertices must be associated to a Grassmann variable $\bar\eta_{a_i}$ (resp. $\eta_{b_i}$) and therefore be in ${\cal A}$ (resp. in ${\cal B}$). The key argument is that, because of the Grassmann nilpotency condition \eqref{defg2}, the paths and cycles must be self-avoiding and pairwise vertex-disjoint!

The term indexed by the flow $(\mathbf{P}, \mathbf{C})$ is equal to $\mathrm{wt}(\mathbf{P}) \,\mathrm{wt}(\mathbf{C})$, up to a sign. By the same argument as in Lemma \ref{detM}, the term associated to $(\mathbf{P}, \mathbf{C})$ differs from the one associated to $(\mathbf{0},\emptyset)$ by a factor $\mathrm{sgn}(\mathbf{C})$. In the latter situation, one can relabel the variables $\eta_{b_i}$ and assume without loss of generality that $P_i$ connects $a_i$ to $b_i$ (for all $i$), and that a factor $\mathrm{sgn}(\mathbf{P})$ 
%has been 
is 
included. 
The only difference with respect to the case studied in Lemma \ref{detM} is that we have now a permutation with $\vert\mathbf{P}\vert = p$ even cycles, yielding an extra factor $(-1)^p$ which cancels the one of formula \eqref{f}. Finally, the sign associated to a general $(\mathbf{P}, \mathbf{C}) \in\mathcal{P}_{{\cal A}{\cal B}}$ is equal to $\mathrm{sgn}(\mathbf{C}) \mathrm{sgn}(\mathbf{P})$, which concludes the proof. 
\end{proof}

%As already mentioned above our Grassmann proof of the Gessel-Viennot formula does not need any involution argument to cancel identical contributions coming with opposite signs from pairs of intersecting paths. All we need is to simply exploit the nilpotency condition satisfied by any Grassmann variable!

%%%%%%%%%%%%%%%%%%%%%%%%%%%%%
%\subsection*{Generalization to graph with cycles}
%%%%%%%%%%%%%%%%%%%%%%%%%%%%%%
%
%The generalization of the Gessel-Viennot formula for graph with cycles was done in \cite{talaska}. The philosophy is identical and self-avoiding path are replaced by self-avoiding flows. Informally, a self-avoiding flow is a collection of self-avoiding paths connecting the sources and the sinks along with a (possibly empty) collection of self-avoiding cycles such that the paths and cycles are pairwise vertex-disjoint - see \cite{talaska} for the formalization.
%
%The Grassmann proof above remains the same, except the fact that $det\; (M^{-1})$ is now non-vanishing (since the graph has cycles). In our Grassmann language, it is this non-vanishing determinant that is equal to the non-trivial denominator of the fraction of the generalization of Thm. $2.5$ of \cite{talaska}.
%
%In \cite{talaska}, the proof is done by constructing a weight-preserving and sign-reversing involution on pairs of paths and cycles. In our case there is no need to construct such an involution. The result trivially follows, once again, from the Grassmann nilpotency condition \eqref{defg2}, %where the respective Grassmann variable 
%as explained above.
%

%%%%%%%%%%%%%%%%%%%%%%%%%%%%%%%%%%%%%%%%%%
\section{Transfer matrix approach}
%%%%%%%%%%%%%%%%%%%%%%%%%%%%%%%%%%%%%%%%%%

In quantum field theory, the path integral represents a space time approach to the time evolution of a system, represented as a sum over paths. Accordingly, the %Lindstr\"om-Gessel-Viennot 
LGV 
lemma is interpreted as the evolution of a system of fermions on a lattice that represents a discrete analogue of space-time. In some instances, it turns out that this evolution can also be described in another formalism based on singling out a time direction in space-time. In our case, this formalism applies to a particular class of graphs which are described below. The sum over paths is then  interpreted as a matrix element of an operator between an initial and a final state which are elements of a Hilbert space constructed as follows. We refer the reader to \cite{Itzykson} for some background  on statistical field theory.

Let us consider $N$ Grassmann variables $\chi_{1},\dots,\chi_{N}$.  
%which we recall obey the anti-commutation relations $\chi_{i}\chi_{j}+\chi_{j}\chi_{i}=0$. 
%We define ${\cal H}$ as the linear span of the $2^{N}$ independent products $\chi_{i_{1}}\dots\chi_{i_{k}}$ with $1\leq i_{1}<\cdots<i_{k}\leq N$. A generic element of ${\cal H}$ is expanded as
%\begin{equation}
%f(\chi)=\sum_{k=0}^{N}\frac{1}{k!}\sum_{1\leq i_{1},\dots,i_{k}\leq N}a_{i_{1}\dots i_{k}}\chi_{i_{1}}\dots\chi_{i_{k}}.
%\end{equation}
%where $a_{i_{1}\dots i_{k}}$ is a array of complex numbers that are antisymmetric under any permutation of its indices $a_{i_{\sigma(1)}\dots i_{\sigma(k)}}=\epsilon(\sigma)a_{i_{1}\dots i_{k}}$. 
The scalar product is defined in analogy with the standard scalar product on holomorphic functions, using an integration over Grassmann variables
\begin{equation}
\langle f,g\rangle=
\int  d\overline{\chi}d\chi
\exp\big(-\overline{\chi}\chi\big)\,
\overline{f}(\overline{\chi})g(\chi)=
\sum_{k=0}^{N}
\frac{1}{k!}
\sum_{1< i_{1},\cdots,i_{k}\leq N}\overline{a}_{i_{1}\dots i_{k}}b_{i_{1}\dots i_{k}}.
\end{equation}
%where we used the shorthand notation $d\overline{\chi}d\chi =d\overline{\chi}_{1}d\chi_{1}\cdots d\overline{\chi}_{N}d{\chi}_{N}$ and $\chi\overline{\chi}=\chi_{1}\overline{\chi}_{1}+\cdots+\chi_{N}\overline{\chi}_{N}$.
 %, with all variables assumed to be anticommutant.
 
Moreover, given an $N\times N$ matrix $\tilde{T}$, one has:
%the latter acts on elements of ${\cal H}$ by
\begin{equation}
\tilde T\!\cdot\!f(\chi)=\sum_{k=0}^{N}\frac{1}{k!}\sum_{1\leq i_{1},\dots,i_{k}\leq N}a_{i_{1}\dots i_{k}}
\big(\sum_{1\leq j_{1}\leq N}\tilde T_{i_{1}j_{1}}\chi_{j_{1}})\dots\big(\sum_{1\leq j_{k}\leq N}\tilde T_{i_{k}j_{k}}\chi_{j_{k}}\big).
\end{equation}
This action can also be written in terms of Grassmann integration as 
\begin{equation}
\tilde T\!\cdot\!f(\chi)=\int  d\overline{\eta}d\eta\,
\exp\big(-\overline{\eta}\eta\big)\,
\exp\big(\overline{\eta}\,\tilde T\, \chi\big)\,
f(\eta),
\end{equation}
%with %all variables anti-commuting and 
where we have used the notation:
$$\overline{\eta}\,\tilde{T}\, \chi=\sum_{i,j=1}^N\overline{\eta}_{i}\tilde{T}_{ij}\chi_{j}.$$ 
Moreover, if $\tilde S$ is another $N\times N$ matrix, %$(ST)\!\cdot\!f(\chi)=S\!\cdot\!(T\!\cdot\!f)(\chi)$ which translates into
\begin{equation}
(\tilde S \, \tilde T)\!\cdot\!f(\chi)=
\int  
d\overline{\psi}d\psi
d\overline{\eta}d\eta\,
\exp\big(-\overline{\eta}\eta\big)
\exp\big(\overline{\eta}\,\tilde S\psi\big)
\exp\big(-\overline{\psi}\psi\big)\,
\exp\big(\overline{\psi}\,\tilde T\chi\big)f(\eta).\label{compose:eq}
\end{equation}
Consider now a sequence of $n$ weighted directed graphs $G_{1},\dots,G_{n}$ each having $N$ vertices labeled by an integer $i\in\left\{1,\dots,N\right\}$. Loops, multiple edges and isolated vertices are allowed. 
%but not multiple edges. 
We denote by $w_{m,ij}$ the weight of an edge oriented from vertex $i$ to vertex $j$ in $G_{m}$, with the convention that the weight vanishes if there is no such an edge. We label the $nN$ vertices of the disjoint union $G_{1}\cup\cdots\cup G_{n}$ by pairs $(i,m)$ where the second index refers to the graph $G_{m}$ and the first one to the vertex $i$ in $G_{m}$.   

We define the graph $G_{1}\rightarrow G_{2}\rightarrow \cdots\rightarrow G_{n}$ by adding $N(n-1)$ edges to the disjoint union $G_{1}\cup\cdots\cup G_{n}$ see Fig. \ref{illustrations:fig}.a. These $N(n-1)$ edges connect the vertex $(i,m)$ to the vertex $(i,m+1)$, for all $m\in\left\{1,\dots , n-1\right\}$ and $i\in\left\{1,\dots , N\right\}$ with a weight $1$. The weighted adjacency matrix of $G_{1}\rightarrow G_{2}\rightarrow \cdots\rightarrow G_{n}$ is given by
\begin{equation}
A_{(i,m),(j,p)}:=
\begin{cases}
w_{i,j}&\text{if }p=m\\
1&\text{if }p=m+1\text{ and $i=j$}\\
0&\text{otherwise}.
\end{cases}
\end{equation}

\begin{figure}
\begin{tabular}{ccc}
\qquad\parbox{2.8cm} {\includegraphics[width=2.8cm]{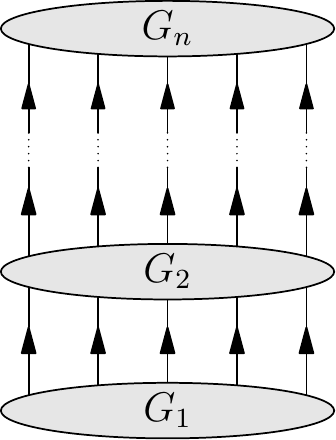}}\smallskip
\qquad
&
\qquad\parbox{6cm}{\includegraphics[width=6cm]{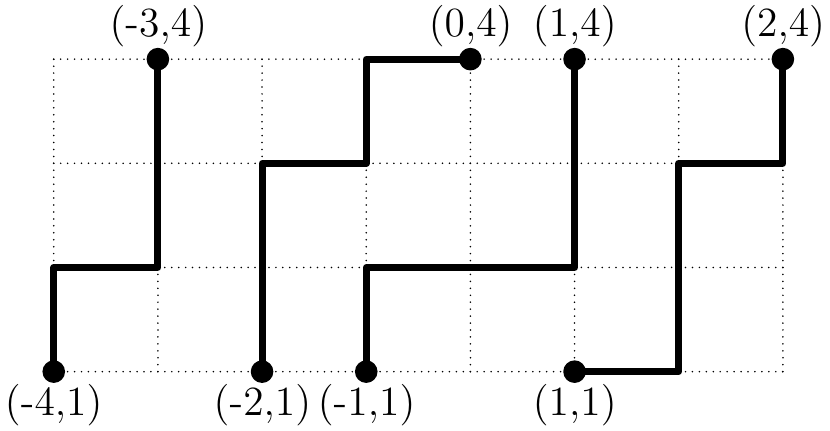}}
\qquad&\qquad
\parbox{2.5cm}{\includegraphics[width=2.5cm]{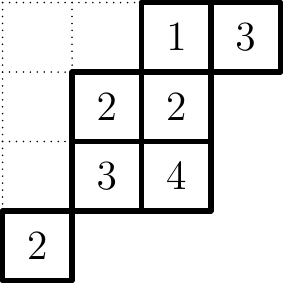}}\\\

\ref{illustrations:fig}.a. 
Graph $G_{1}\rightarrow G_{2}\rightarrow\cdots\rightarrow G_{n}$& 
\ref{illustrations:fig}.b. 
Non intersecting lattice paths & 
1.c. 
Skew Young table
\end{tabular}
\caption{Some illustrations}
\label{illustrations:fig}
\end{figure}
The previous construction is motivated by the following theorem, relating $k$ non intersecting paths in $G_{1}\rightarrow G_{2}\rightarrow \cdots\rightarrow G_{n}$, starting at vertices ${\cal A}\in G_{1}$ and ending at vertices ${\cal B}\in G_{n}$, to a $k\times k$ minor in a $N\times N$ matrix constructed using the weighted adjacency matrices $A_{i}$ of $G_{i}$.
\begin{theorem}[]
One has
\begin{multline}
\sum_{\text{non intersecting paths ${\cal P}_{1},\dots,{\cal P}_{k}$ ${\cal A}\rightarrow {\cal B}$}\atop \text{ and cycles ${\cal C}_{1},\dots,{\cal C}_{r}$ in 
     $G_{1}\rightarrow G_{2}\rightarrow \cdots\rightarrow G_{n}$}}
     (-1)^{\epsilon({\cal A},{\cal B}) }
(-1)^{r}w({\cal P}_{1})\cdots w({\cal P}_{k})
w({\cal C}_{1})\cdots w({\cal C}_{r})\\
=\qquad\det(1-A_{n})\cdots \det(1-A_{1}) 
\det \big[(1-A_{n})^{-1}\cdots (1-A_{1})^{-1}\big]_{{\cal A}{\cal B}},
\label{GVtransfert1:eq}
\end{multline}
with $\epsilon({\cal A},{\cal B})$ the signature of the permutation of the labels of the vertices in ${\cal B}$  with respect to those in ${\cal A}$ and $\det M_{{\cal A},{\cal B}}$ is the determinant restricted to the lines corresponding to ${\cal A}$ and columns to ${\cal B}$.
%\\
%&=\qquad\det \big[(1-A_{n})\cdots (1-A_{1})\big]_{{}^{c}{\cal A}{}^{c}{\cal B}}
%\label{GVtransfert2:eq}
\label{GVtransfert:thm}
\end{theorem}
\begin{proof}
The result 
is proved
by induction on $n$. For $n=1$, the statement corresponds to Theorem  \ref{GV:thm}. Then, one passes from $n$ to $n+1$ by integrating pairs of variables between vertices $(i,m)$ and $(i,m+1)$ and the use of  \eqref{compose:eq}. 
%We establish this result by induction on $n$. For $n=1$, this is nothing but the Lindstr\"om-Gessel-Viennot Lemma. Then, assume that the result holds at order $n-1$. Consider the left-hand side at order $n$.We first observe that any cycle belongs to one of the $G_{m}$'s since all edges between different $G_{m}$ are oriented from $m$ to $m+1$. Then, the cycles in $G_{1}\rightarrow G_{2}\rightarrow \cdots\rightarrow G_{n}$ partition into cycles in $G_{1}\rightarrow G_{2}\rightarrow \cdots\rightarrow G_{n-1}$ and cycles in $G_{n}$. Then, decompose a path in $G_{1}\rightarrow G_{2}\rightarrow \cdots\rightarrow G_{n}$ as a path in $G_{1}\rightarrow G_{2}\rightarrow \cdots\rightarrow G_{n-1}$ concatenated with a path in $G_{m}$. Using the result at order $n-1$ and order $n$, we write the left-hand side as a sum of products of two minors. The integration over the Grassmann variables connecting $G_{1}\rightarrow G_{2}\rightarrow \cdots\rightarrow G_{n-1}$ and $G_{n}$ yields the result at order $n$ using \eqref{compose:eq}. 
\end{proof}

In statistical physics, a homogeneous term of degree $k$ in ${\cal H}$ represents a state of $k$ fermions occupying the vertices of $G_{m}$ at time $m$. The anti-commutation relations express Pauli exclusion principle that states that two fermions cannot occupy the same vertex. The operator $(1-A_{m})^{-1}$ (multiplied by a power of its determinant) transforms this state into another $k$ fermion state at time $m+1$, on the vertices of $G_{m+1}$. Thus, $T_{m}$ represents a discrete time evolution;
this matrix is known in physics as the {\it transfer matrix}. 

The interest of this result comes from the evaluation of the sum over paths by a minor in an $N\times N$ matrix instead of an $nN\times nN$ matrix as would result from an application of the 
%Lindstr\"om-
%Gessel-Viennot 
LGV 
lemma. In the next two sections we show how this result can be used in the theory of Schur functions. Other related applications of fermionic techniques can be found in \cite{Nadeau} and \cite{Zinn}. 
%as is illustrated in the next two sections.

%%%%%%%%%%%%%%%%%%%%%%%%%%%%%
\section{An application to Schur functions}
%%%%%%%%%%%%%%%%%%%%%%%%%%%%%
\label{transfert:sec}

Given an integer $k$, a partition of $k$ is a decreasing sequence $\lambda_{1}\geq\dots\geq\lambda_{r}$ of $r$ integers such that $\lambda_{1}+\cdots+\lambda_{r}=k$. A partition is conveniently represented by a Young diagram denoted $\lambda$ and made of $r$ left justified  rows, the $k^{\text{th}}$ row containing $\lambda_{k}$ boxes, with the longer rows on the top of the shorter ones. We set $|\lambda|:=\lambda_{1}+\dots+\lambda_{r}$. 

Given a second Young diagram $\mu$ with $r'$ rows, we write $\mu\leq \lambda$ if $r'\leq r$ and if for all $i\left\{1,\dots,r\right\}$, $\mu_{i}\leq \lambda_{i}$. When $\mu\leq\lambda$, the skew Young diagram $\lambda/\mu$ is constructed by removing the $\mu_{i}$ first left boxes in the line $i$ of $\lambda$ for all $i$. We also consider the empty Young diagram and $\lambda/\emptyset=\lambda$ while $\lambda/\lambda=\emptyset$.  We further set $\mu_{i}=0$ for $i\geq r'$.

A semi standard (skew) Young tableau %${\cal T}$ 
(SSYT) of shape $\lambda/\mu$ is a filling of the Young diagram $\lambda/\mu$ by some integers in $\left\{1,\dots,n\right\}$ in such a way that they are increasing along the columns and non decreasing along the rows. To each of these integers we associate an  indeterminate $x_{m}$ and the Schur function is defined as
\begin{equation}
s_{\lambda/\mu}(x):=\sum_{\text{skew Young Tableau %${\cal T}$
}\atop\text{ of shape $\lambda/\mu$}}
\prod_{1\leq m\leq n} x_{m}^{k_{m}},
\end{equation}
where $k_{m}$ is number of times the integer $m$ appears in the SSYT, see Fig.  \ref{illustrations:fig}.c. 
%Schur functions are fundamental objects in combinatorics and representation theory. In particular, they provide a convenient basis of the algebra of symmetric polynomial, although it is not obvious from the previous definition that they indeed are symmetric.

%Thanks to the seminal work of Gessel and Viennot \cite{GV}, 
It is known (see \cite{GV})  that $s_{\lambda/\mu}(x)$ can be constructed using $r$ non intersecting lattice paths as follows. Define a graph $G$  with vertices labelled $(i,m)$ with $i$ and $m$ positive integers and oriented  edges from $(i,m)$ to $(i+1,m)$ and from $(i,m)$ to $(i,m+1)$. 
The graph $G$ is conveniently visualized as a two dimensional square lattice with arrows pointing upwards and rightwards. Although infinite, at any stage of the computation only a finite number of vertices are involved. We leave the precise range of $i$ unspecified for notational convenience and assume $1\leq m\leq n$ unless otherwise stated. Then, the skew Schur functions can be written as a sum over $r$ non intersecting paths on G,
\begin{equation}
s_{\lambda}(x)=\sum_{\text{non intersecting lattice paths ${\cal P}_{1},\dots,{\cal P}_{r}$}\atop
{\cal P}_{i}:\,(\mu_{i}-i+l,1)\rightarrow (\lambda_{i}-i+l,n)}W({\cal P}_{1})\cdots W({\cal P}_{i}).
\end{equation}
where $l$ is a global translation parameter that does not affect the result, because of translation invariance. 
The weight of a path is again given by the product of the weight of its edges. The weight of an horizontal edge from $(i,m)$ to $(i+1,m)$ is $x_{m}$ and the weight of all vertical edges is 1.

The graph $G$ can be written as $G=G_{1}\rightarrow G_{2}\rightarrow \cdots\rightarrow G_{n}$ with all $G_{m}$ isomorphic to a one dimensional lattice with edges oriented to the right, i.e. from $i$ to $i+1$. The weighted adjacency matrix is made of right translations $T$ (defined by $T_{ij}=1$ if $j=i+1$ and $0$ otherwise) multiplied by $x_{m}$, such that $1-A_{m}=1-x_{m}T$.
%\begin{pmatrix}
%1&x_{m}&0&0&0&\cdots \\
%0&1&x_{m}&0&0&\cdots\\
%0&0&1&x_{m}&0&\\
%\vdots&\vdots&\ddots&\ddots
%\end{pmatrix}
%\end{equation}
Its inverse reads $(1-A_{m})^{-1}=\sum_{p\geq 0}(x_{m})^{p}\,T^{p}$. One then has 
\begin{equation}
(1-A_{n})^{-1}\cdots(1-A_{1})^{-1}=\sum_{k}h_{k}(x)T^{k}\label{inverse:eq},
\end{equation}
with $h_{k}(x)$ the complete symmetric functions of $x_{1},\dots,x_{n}$ of degree $k$,
\begin{equation}
h_{k}(x)=\sum_{k_{1}+\cdots k_{n}=k}x_{1}^{k_{1}}\cdots x_{n}^{k_{n}}.
\end{equation}
We can apply Theorem \ref{GVtransfert:thm} (with all $G_{m}$ acyclic so that there is no contribution of cycles) and equation \eqref{GVtransfert1:eq} yields
%\begin{equation}
%s_{\lambda}(x)=
%\det \big[h_{j-i}(x)\big]_{i=
%1,2,\dots,k\big|j=1+\lambda_{k},2+\lambda_{k-1},\dots,k+\lambda_{1}}
%\end{equation}
%with the convention that $h_{j-i}(x)=0$ if $i>j$. This is nothing but the first Jacobi-Trudi identity, after reversing the order of the lines and of the columns,
the celebrated Jacobi-Trudi identity
\begin{equation}
s_{\lambda/\mu}(x)=
\det \big(h_{\lambda_{j}-\mu_{i}+i-j}(x)\big)_{1\leq,i,j\leq r},\label{JacobiTrudi1:eq}
\end{equation}
for a skew partition with $r$ rows, with the convention that $h_{j-i}(x)=0$ if $i>j
$.

From a physical point of view, we may associate to a partition an element of ${\cal H}$ defined by $|\lambda\rangle=\chi_{{}\atop\lambda_{1}-1+l}\cdots \chi_{{}\atop\lambda_{r}-r+l}$. Introducing $U(x)=(1-x_{n}T)^{-1}\cdots(1-x_{1}T)^{-1}$, Schur functions are transition amplitudes between two such states, $s_{\lambda/\mu}(x)=\langle \lambda|U(x)|\mu\rangle$, which is non zero only if $\mu\leq\lambda$. 

If we separate the variables $x$ into two disjoints sets denoted $x'$ and $x''$, one has:
$U(x)=U(x')U(x'')$. This comes from the fact that all these operators commute. The relation 
\begin{equation}
\langle\lambda|U(x)|\mu\rangle=\sum_{%\nu\text{ partition}
\mu\leq\nu\leq\lambda}
\langle\lambda|U(x')|\nu\rangle\langle\nu|U(x'')|\mu\rangle
\end{equation}
then leads to the convolution identity:
\begin{equation}
s_{\lambda/\mu}(x)=\sum_{%\nu\text{ partition} 
\mu\leq\nu\leq \lambda}
s_{\lambda/\nu}(x')s_{\nu/\mu}(x'').
\end{equation}
This identity 
follows from
%is immediately seen to be true using 
 the LGV
% Gessel-Viennot 
 lemma. From a lattice point of view, this is a vertical composition. In the next section, we will derive an horizontal composition from the multiplication law 
 $$U^{a+b}(x)=U^{a}(x)U^{b}(x).$$

%The second Jacobi-Trudi identity is recovered from equation \eqref{GVtransfert2:eq} with 
%\begin{equation}
%(1-A_{n})\cdots(1-A_{1})=\sum_{p}e_{p}(x)\tau^{p}
%\end{equation}
%where $e_{p}(x)$ are the elementary symmetric functions of $x_{1},\dots,x_{n}$
%\begin{equation}
%e_{p}(x)=\sum_{1\leq i_{1}<\cdots<i_{p}\leq n}x_{1}^{i_{1}}\cdots x_{n}^{i_{n}}.
%\end{equation}

\section{A one parameter extension of Schur polynomials}

%To begin with, l
Let us introduce the following symmetric polynomials 
 \begin{equation}
S_{k}(a,x)=
\sum_{k_{1}+\cdots k_{n}=k}x_{1}^{k_{1}}\dots x_{n}^{k_{n}}
\prod_{1\leq m\leq n}\frac{a(a+1)\dots(a+k_{m}-1)}{k_{m}!}.
\end{equation}
For $a=1$ we recover the complete homogeneous polynomials $S_{k}(1,x)=h_{k}(x)$. For example,  %and 
%and for $a=-1$ the elementary symmetric polynomials, $S_{k}(-1,x)=(-1)^{k}e_{k}(x)$.Let us give a few examples for low values of $k$,
\begin{align}
S_{1}(a,x)&=
a
\sum_{1\leq m\leq n} x_{m},
\qquad\qquad
S_{2}(a,x)=
\frac{a(a+1)}{2}\sum_{1\leq m\leq n} x^{2}_{m}
+a^{2}\sum_{1\leq p<m\leq n} x_{m}x_{p},\\
S_{3}(a,x)&=
\frac{a(a+1)(a+2)}{6}\sum_{1\leq m\leq n} x^{3}_{m}+
\frac{a^{2}(a+1)}{2}\!\!\!\!\sum_{1\leq p<m\leq n}\!\!\!\! (x^{2}_{m}
x_{p}+x_{m}x_{p}^{2})\,\,+\,\,a^{3}\!\!\!\!\sum_{1\leq q<p<m\leq n}\!\!\!\! x_{m}x_{p}x_{q}.
\end{align}
%These symmetric polynomials bear some similarities with the Jack polynomials, although much simpler. For instance, they decompose into Jack polynomials for simple partitions of $k$. 
These polynomials appear in the expansion of  $U^{a}(x)=(1-x_{n}T)^{-a}\cdots (1-x_{1}T)^{-a}$, generalizing \eqref{inverse:eq},
\begin{equation}
U^{a}(x)=
\sum_{ k\geq 0}
S_{k}(a,x) T^{k},
\end{equation}
which follows from writing $(1-x_{m}T)^{-a}=\frac{1}{\Gamma(a)}\int_{0}^{\infty}dt_{m}t_{m}^{a-1}
\exp-t_{m}(1-x_{m}T)$.
%\begin{equation}
%(1-x_{n}T)^{a}\cdots (1-x_{1}T)^{a}=
%\int_{0}^{\infty}dt_{1}
%\int_{0}^{\infty}dt_{n}
%\sum_{k_{1},\dots,k_{n}}
%\prod_{1\leq m\leq n}
%\frac{t_{m}^{a+k_{m}-1}\exp-t_{m}}{\Gamma(a)k_{m}!}
%\,x_{1}^{k_{1}}\cdots x_{n}^{k_{n}}\,T^{k_{1}+\dots+k_{n}},
%\end{equation}
%and perform the integral over $t_{1},\dots,t_{n}$. This extension equation \eqref{inverse:eq} which we recover for $a=1$.
%\end{proof}
Using $U^{a}(x)=(1-x_{n}T)^{-a}\cdots (1-x_{1}T)^{-a}$ in equation \eqref{GVtransfert1:eq} instead of $U(x)=(1-x_{n}T)^{-1}\cdots (1-x_{1}T)^{-1}$ leads to a one parameter generalization of the Schur function. The latter are defined by replacing the $h_{k}(x)$ by $S_{k}(a,x)$ in the Jacobi-Trudi identity \eqref{JacobiTrudi1:eq}.
\begin{definition}[One parameter extension of Schur polynomials]
Let 
\begin{equation} 
s_{\lambda/\mu}(a,x):=
\det \big(S_{\lambda_{j}-\mu_{i}+i-j}(a,x)\big)_{1\leq i,j\leq r}.
\label{extendedSchur:eq}
\end{equation}
We use here the convention $S_{0}(a,x)=1$ and $S_{k}(a,x)=0$ for $k<0$.
\end{definition}
Schur functions are recovered for $a=1$, $s_{\lambda/\mu}(1,x)=s_{\lambda/\mu}(x)$. Theorem  \ref{GVtransfert:thm} then implies that $S_{\lambda/\mu}$ can also be written as a sum over $r$ non intersecting lattice paths for a skew diagram with $r$ rows. However, since we use $(1-x_{m}T)^{a}$ instead of $(1-x_{m}T)$, for $j>i$ there is an edge from $(i,m)$ to $(j,m)$ weighted by 
\begin{equation}
w_{(i,m)\rightarrow(j,m)}=(-1)^{j-i+1}\frac{a(a-1)\dots (a-j+i+1)}{(j-i)!}\,x_{m}^{j-i}.  
\end{equation}
In that case, the paths $(i,m)\rightarrow(i,p)\rightarrow(j,p)\rightarrow(j,q)$ and $(k,m)\rightarrow(k,q)$ for $i<k<j$ do not intersect but contribute with an extra $-1$ because the order of their endpoints have been reversed. 

%Let us illustrate this construction on the example of the partition $(2,1)$.

\begin{example}[$s_{(2,1)}(a,x)$ as a sum over paths]

The paths contributing to $s_{(2,1)}(a,x)$ join vertices $(1,1)$ and $(2,1)$ on on side 
and $(2,n)$ and $(4,n)$ on the other side. 

\begin{center}
\begin{tabular}{|c|c|}
\hline
\begin{minipage}{10cm}

\begin{center}
\smallskip
$(1,1)\rightarrow (1,m)\rightarrow(2,m)\rightarrow (2,n)$\\
 $(2,1)\rightarrow (2,p)\rightarrow(3,p)
\rightarrow(3,q)\rightarrow(4,q)\rightarrow(4,n)$
\end{center}
\end{minipage}&
\begin{minipage}{6cm}
\begin{center}
\smallskip
${\displaystyle
a^{3}\sum_{1\leq p<m\leq n,\atop 1\leq p\leq q\leq n}x_{m}x_{p}x_{q}}$
\smallskip
\end{center}

\end{minipage}\\
\hline

%\begin{minipage}{10cm}
%\smallskip
%\begin{center}
%$(1,1)\rightarrow (1,m)\rightarrow(2,m)\rightarrow (2,n)$\\
% $(2,1)\rightarrow (2,p)\rightarrow(3,p)
%\rightarrow(4,p)\rightarrow(4,n)$\end{center}
%\end{minipage}&
%\begin{minipage}{6cm}
%\begin{center}
%\smallskip
%${\displaystyle
%a^{3}\sum_{1\leq p <m\leq n}x_{m}x_{p}^{2}}$
%\smallskip
%\end{center}

%\end{minipage}\\
%\hline

\begin{minipage}{10cm}
\smallskip
\begin{center}
$(1,1)\rightarrow (1,m)\rightarrow(2,m)\rightarrow (2,n)$\\
 $(2,1)\rightarrow (2,p)\rightarrow(4,p)
\rightarrow(4,n)$
\end{center}
\end{minipage}&
\begin{minipage}{6cm}
\begin{center}
\smallskip
${\displaystyle
-\frac{a^{2}(a-1)}{2}\sum_{1\leq p <m\leq n}x_{m}x_{p}^{2}}$
\smallskip
\end{center}

\end{minipage}
\\
\hline
\begin{minipage}{10cm}
\smallskip
\begin{center}
$(1,1)\rightarrow (1,m)\rightarrow(4,m)\rightarrow (4,n)$\\
 $(2,1)\rightarrow (2,n)$
\end{center}
\end{minipage}&
\begin{minipage}{6cm}
\begin{center}
\smallskip
${\displaystyle
-\frac{a(a-1)(a-2)}{6}\sum_{1\leq m\leq n}x_{m}^{3}}$
\smallskip
\end{center}
\end{minipage}
\\
\hline
\begin{minipage}{10cm}
\smallskip
\begin{center}
$(1,1)\rightarrow (1,m)\rightarrow(3,m)\rightarrow(3,p)\rightarrow(4,p)\rightarrow (4,n)$\\
 $(2,1)\rightarrow (2,n)$
\end{center}
\end{minipage}&
\begin{minipage}{6cm}
\begin{center}
\smallskip
${\displaystyle
+\frac{a^{2}(a-1)}{2}\sum_{1\leq m \leq p\leq n}x_{m}^{2}x_{p}}$
\smallskip
\end{center}
\end{minipage}
\\
\hline
\end{tabular}
\end{center}
For Schur functions, the last three contributions are absent $(a=1)$, since they involve horizontal segments of length $2$ and $3$. In the last two rows there is a extra sign because of the interchange of endpoints. 
\begin{align}
s_{\includegraphics[width=0.3cm]{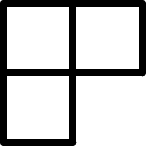}}(a,x)=
\det\begin{pmatrix}
S_{2}(a,x)&1\\S_{3}(a,x)&S_{2}(a,x)
\end{pmatrix}
=\frac{a(a^{2}-1)}{3}\sum_{1\leq m\leq n}x_{m}^{3}
+\,\,\,a^{2}\!\!\!\sum_{1\leq p< m\leq n \atop 1\leq p\leq q \leq n}\!\!\!\!x_{m}x_{p}x_{q}.
\end{align}
\end{example}

The main interest of this extension of Schur polynomials is the following convolution identity:

\begin{theorem}[Convolution identity] One has
\begin{equation}
s_{\lambda/\mu}(a+b,x)
=
\sum_{\nu\,\text{partition}\atop
\mu\leq \nu\leq \lambda}
s_{\lambda/\nu}(a,x)
s_{\nu/\mu}(b,x).
\label{convolutionSchur:eq}.
\end{equation}
Note that, for the empty partition, one has: $s_{\emptyset}(a,x)=1$.
\end{theorem}

\begin{proof}
The proof relies on the multiplication law $U^{a}(x)U^{b}(x)=U^{a+b}(x)$.
%\begin{equation}
%(1-x_{n}T)^{-(a+b)}\cdots (1-x_{1}T)^{-(a+b)}
%=
%(1-x_{n}T)^{-a}\cdots (1-x_{1}T)^{-a}\,
%(1-x_{n}T)^{-b}\cdots (1-x_{1}T)^{-b},\label{compose:eq}
%\end{equation}
This translates to
\begin{equation}
S_{k}(a+b,x)=\sum_{p+q=k}S_{p}(a,x)S_{q}(b,x).
\end{equation}
The result then follows from the expansion of the determinant in \eqref{extendedSchur:eq}, expansion which uses the Cauchy-Binet formula.

%Then, we associate the partitions $\lambda$, $\mu$  and $\nu$ with some vectors $|\lambda\rangle$, $|\mu\rangle$ and $|\nu\rangle$ in the Hilbert space ${\cal H}$ consisting in monomials of Grassmann variables (see section \ref{transfert:sec}). The multiplication law \label{compose:eq} then translates into the matrix multiplication which yields \eqref{convolutionSchur:eq}, 
%\begin{multline}
%\langle \lambda|(1-x_{n}T)^{-a-b}\cdots (1-x_{1}T)^{-a-b}|\mu\rangle
%=\\
%\sum_{\nu\atop
%\mu\leq \nu\leq \lambda}
%\langle\lambda|
%(1-x_{n}T)^{-a}\cdots (1-x_{1}T)^{-a}|\nu\rangle\,\langle \nu|
%(1-x_{n}T)^{-b}\cdots (1-x_{1}T)^{-b}|\mu\rangle\label{compose:eq}.
%\end{multline}
\end{proof} 

\begin{example}
The convolution identity for $(2,1)$ reads
\begin{align}
s_{\includegraphics[width=0.3cm]{partition21}}(a+b,x)
=
s_{\includegraphics[width=0.3cm]{partition21}}(a,x)+
s_{\includegraphics[width=0.15cm]{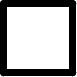}}(a)s_{\includegraphics[width=0.15cm]{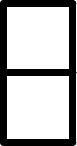}}(b)(x)
+
s_{_{\includegraphics[width=0.3cm]{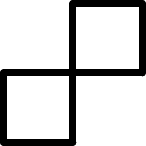}}}(a,x)s_{\includegraphics[width=0.15cm]{partition1}}(b,x)
+
s_{\includegraphics[width=0.15cm]{partition1}}(a,x)s_{\includegraphics[width=0.3cm]{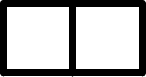}}(b,x)
+s_{\includegraphics[width=0.3cm]{partition21}}(b,x).
\end{align}
\end{example}

Other identities satisfied by $s_{\lambda}(a,x)$ can easily be proven. For example, for the conjugate diagrams (obtained by symmetry with respect to the main diagonal), one has: 
$$s_{\lambda^{\ast}/\mu^{\ast}}(a,x)=(-1)^{|\lambda|-|\mu|}s_{\lambda/\mu}(-a,x).$$
%The convolution identity may be used to find some identities relating the Schur functions by using $s_{a}(x)$ only as intermediate objects.  
%Finally, let us mention that it would also be interesting to
%find the relation with the Jack symmetric polynomials, studied in relation with combinatorics of Young diagrams by Stanley in \cite{Stanley}.

%Moreover, if we expand to first order in $n$, we find a system of linear differential equations 
%\begin{equation}
%\frac{ds_{\lambda/\mu}}{da}(a,x)
%\sum_{\nu\,\text{partition}\atop
%\mu\leq \nu\leq \lambda}
%s_{\lambda/\nu}(a,x)
%L_{\nu/\mu}(x).
%\label{convolutionSchur:eq}
%\end{equation}
%$L_{\nu/\mu}(x)$ is a triangular matrix only involves paths having a single horizontal segment. Solving this system leads to an expression of the Schur functions in terms of power sums.

\section*{Acknowledgements}
TK and AT are partially supported by the grant ANR JCJC ``CombPhysMat2Tens".
AT is partially supported by the grant PN 
16 42 01 01/2016.
%09 37 01 02. 
SC is supported by the grant ANR JCJC ``CombPhysMat2Tens".
AT thanks JF Marckert for carefully reading the first part of this paper.

%%%%%%%%%%%%%%%%%%%%%%%%%%%%%%%%%%%%%%%%%%%%%%%%%%%%%%%
 %\bibliographystyle{abbrvnat}
%{plain} 
% \bibliography{myBibFile} 
% If you use BibTeX to create a bibliography
% then copy and past the contents of your .bbl file into your .tex file
%{abbrvnat}

%\bibliographystyle{abbrvnat}
% use the following instead if you encounter problems 
%\bibliographystyle{plain}
%\bibliography{eurocomb}
%\label{sec:biblio}
%\nocite{*}

%\bigskip

%\noindent
%Sylvain Carrozza\\
%{\it\small 
%LaBRI, Universit\'e Bordeaux, Talence, France, EU }

%\medskip

%\noindent Thomas Krajewski \\
%{\it\small  CPT, Aix-Marseille Universit\'e, Marseille, France, EU}

%\noindent Iain Moffatt, 
%{\it\small  Royal Holloway University of London, Egham, UK, EU}

%\medskip

%\noindent
%Adrian Tanasa\\
%{\it\small 
%LaBRI, Universit\'e Bordeaux, Talence, France, EU}\\
%{\it\small H. Hulubei Nat. Inst.  Phys.  Nucl. Engineering,
%Magurele, Romania, EU}\\
%{\it\small 
%IUF Paris, France, EU }\\	
	
\end{document}